\documentclass[12pt]{article}
\usepackage[utf8]{inputenc}
\usepackage{amssymb}
\usepackage[margin=2.5cm]{geometry}

\usepackage{amsmath,amsthm,graphicx,float}
\usepackage{amssymb}
\usepackage{xcolor}
\usepackage{cleveref}
\usepackage{subcaption}
\usepackage{enumitem}
\usepackage{dsfont}
\usepackage{caption}
\usepackage{verbatim}
\usepackage{longtable}

\newtheorem{theorem}{Theorem}[section]
\newtheorem{problem}[theorem]{Problem}
\newtheorem{lemma}[theorem]{Lemma}
\newtheorem{proposition}[theorem]{Proposition}
\newtheorem{corollary}[theorem]{Corollary}

\newtheorem{procedure}[theorem]{Procedure}
\newtheorem{observation}[theorem]{Observation}

\newtheorem{definition}[theorem]{Definition}

\newcommand{\single}{\chi_{\nleftrightarrow}}
\newcommand{\ch}{\mathrm{ch}}

\newcommand{\sing}{\chi_{\nleftrightarrow}}

\newcommand{\cv}{\{c_v\}}

\newcommand{\setdiv}{\text{ }|\text{ }}

\newcommand{\E}{\mathbb{E}}
\renewcommand{\P}{\mathbb{P}}
\newcommand{\Var}{\text{Var}}

\newcommand{\Z}{\mathbb{Z}}

\begin{document}
\title{Single conflict coloring and palette sparsification of uniform 
hypergraphs\footnote{This paper is partially based on the thesis written 
by the second author under the supervision of the first author}}

\author{
{\sl Carl Johan Casselgren}\thanks{{\it E-mail address:} 
carl.johan.casselgren@liu.se}\\ 
Department of Mathematics \\
Link\"oping University \\ 
SE-581 83 Link\"oping, Sweden
\and
{\sl Kalle Eriksson}\thanks{{\it E-mail address:} 
kalle.eriksson@math.su.se}\\ 
Department of Mathematics \\
Stockholm University \\ 
SE-106 91 Stockholm, Sweden
}

\maketitle


\begin{abstract}
	We introduce and investigate single conflict coloring in the setting of 
	$r$-uniform hypergraphs.
	We establish some basic properties of this hypergraph coloring model
	and study a probabilistic model of single conflict coloring where the conflicts for each 
	edge are chosen randomly; in particular,
	we prove a sharp threshold-type result for complete graphs and 
	establish a sufficient condition for
	single conflict colorability of $r$-uniform hypergraphs in this model.
	Furthermore, we obtain a related palette sparsification-type result for general list coloring of linear
	uniform hypergraphs (i.e.\ uniform hypergraphs where any two edges share at most one common vertex).
	Throughout the paper we pose several questions and conjectures.
\end{abstract}

\noindent
{\bf Keywords:}  Hypergraph coloring; Single conflict coloring; Palette sparsification

\noindent
{\bf Mathematics Subject Classification:}  05C15

\section{Introduction}
Being first introduced in the 70's \cite{ERT, Vizing}, list coloring is a well-established area of hypergraph coloring.
Formally, given a hypergraph $G$, we assign to each vertex $v$ of $G$ a set
        $L(v)$ of colors; $L$ is called a 
   \emph{list assignment} for $G$ and
    the sets $L(v)$ are referred
    to as \emph{lists} or \emph{color lists}.
    If there is a proper coloring $\varphi$ of $G$ (i.e.~no edge is monochromatic) such that 
    $\varphi(v) \in L(v)$ for all $v \in V(G)$, then
	$G$ is \emph{$L$-colorable} and $\varphi$
    is called an \emph{$L$-coloring} or simply a {\em list coloring}. 
	Furthermore, $G$ is
    called \emph{$k$-choosable} if it is $L$-colorable
    for every list assignment $L$ where all lists have size at least $k$. 
	The minimum $k$ such that $G$ is $k$-choosable is the 
	{\em list-chromatic number} (or {\em choice number})  $\ch(G)$ of $G$.

Many hypergraph coloring models  translate naturally to the list coloring setting, and numerous variants of list coloring are studied in the literature. Our primary interest here, {\em single conflict coloring}, is a generalization
of several other (hyper-)graph list coloring invariants, so let us begin by briefly reviewing some related coloring models.

Given an edge coloring $f$ of a hypergraph $G$, a coloring $c : V(G) \to \mathbf{Z}$ is {\em adapted} to $f$ if no 
edge $e$ satisfies that every vertex contained in $e$ has the the color $f(e)$ under $c$. 
Adapted coloring was introduced by
Hell and Zhu \cite{HellZhu}, who
among other things proved that a graph is adaptably $2$-colorable
if and only if it contains an edge whose removal yields a bipartite graph. 
They also proved general bounds on the 
{\em adaptable chromatic number} $\chi_a(G)$,
that is, the least number of colors needed for an adaptable coloring of a graph $G$.

Building on the notion of adaptable coloring, Kostochka and Zhu
\cite{KostochkaZhu} considered the list coloring version of adaptable coloring for graphs and hypergraphs, 
where the vertices are assigned lists, and we seek
a list coloring that is adapted to a given edge coloring. The minimum size of the lists which ensures the existence of such an
adapted list coloring is the {\em adaptable choice number}
$\ch_{ad}(G)$ (or {\em adaptable choosability}) of a hypergraph $G$. 
Kostochka and Zhu proved a general upper bound on this invariant and also characterized graphs with adaptable choosability $2$.

Another variant of list coloring is so-called
choosability with separation, introduced by Kratochvil, Tuza and Voigt
\cite{KratochvilTuzaVoigt}. In this model, we require that color lists of  vertices contained in a common edge
have bounded intersection,
and ask for the minimum size of the lists which guarantees existence of a proper list coloring of the hypergraph in question.
Single conflict coloring is related to the variant where lists of vertices contained in a common edge have
at most one common color, often referred to as  {\em separation choosability}.
Denote by $\ch_{sep}(G)$ the smallest size of the lists which ensure that $G$ is colorable
from such lists, provided that lists of vertices contained in a common edge intersect in at most one color. This
particular variant was specifically studied in e.g. \cite{EsperetKangThomasse, FurediKostochkaKumbhat}.

Generalizing both the notion of adaptable choosability and separation choosability
\cite{KratochvilTuzaVoigt},
single conflict coloring of graphs was introduced by Dvorak et al.\ 
\cite{DvorakEsperetKangOzeki}. 
A {\em local $k$-partition} of a graph $G$ is a collection $\{c_v\}_{v \in V(G)}$ of maps 
	$c_v : E(v) \to \{1,\dots, k\}$, where $E(v)$ denotes the set of edges incident
	with $v$. We think of $c$ as a conflict map, and for each
	edge $e = xy \in E(G)$, $c$ yields a conflict-pair 
	$(c_{x}(e), c_{y}(e))$. Given such a local $k$-partition $\{c_v\}$,
	$G$ is {\em conflict $\{c_v\}$-colorable} if there is a map 
	$\varphi: V(G) \to \{1,\dots, k\}$ such that no edge $e = xy$ 
	satisfies that $\varphi(x) = c_{x}(e)$ and $\varphi(y) = c_{y}(e)$.
	The {\em single conflict chromatic number} $\single(G)$ of
	$G$ is the smallest $k$ such that
	$G$ is conflict $\{c_v\}$-colorable for every local $k$-partition $\{c_v\}$.

	Dvorak and Postle \cite{DvorakPostle} introduced correspondence coloring
	(also known as DP-coloring) as a generalization of list coloring. Roughly, the idea
	is to replace every vertex $v$ in a hypergraph $G$
	by a set $I_v$ of $k$ isolated vertices, and to replace every edge of $G$ by a (hyper-)matching.
	(The vertices in the added independent sets correspond to the colors in a list of size $k$.)
	If there is a transversal (i.e.\ an independent set intersecting each $I_v$ in exactly one vertex)
	in the resulting hypergraph for any choice of the matchings joining
	independent sets, then the graph is 
	{\em $k$-DP-colorable}.
	The {\em DP-chromatic number $\chi_{DP}(G)$} is the minimum $k$ such that
	$G$ is $k$-DP-colorable\footnote{See \cite{Schweser} for a formal definition.}.

	In this paper we introduce and study
	single conflict coloring in the setting
	of $r$-uniform hypergraphs.\footnote{A formal definition is given at the beginning of Section 2.}
	We establish some basic properties and relate the single conflict chromatic number to other hypergraph
	coloring invariants. 
	In particular, based on a construction in \cite{KostochkaZhu}, for every pair of integers $r,d$,
	we give an example of an $r$-uniform hypergraph $H_{r,d}$ such that
	$$\ch(H_{r,d}) = \ch_{sep}(H_{r,d}) = \ch_{ad}(H_{r,d}) = \single(H_{r,d}) = 
	\chi_{DP}(H_{r,d}) =d.$$

	Next, inspired by the probabilistic model of DP-coloring of graphs from random covers studied in
	\cite{BernshteynDominikKaulMudrock},
	we consider a random model of single conflict coloring of $r$-uniform hypergraphs, which we define
	formally in Section 3. As in \cite{BernshteynDominikKaulMudrock} we analyze the probability
	of the existence of a single conflict coloring in terms of degeneracy; we also prove a sharp
	threshold type result for complete graphs, and conjecture that an analogous statement is
	true for complete $r$-uniform hypergraphs.
	
	Finally, we consider the related problem of palette sparsification of ordinary list coloring
	of $r$-uniform hypergraphs.
	Assadi, Chen, and Khanna \cite{AssadiChenKhanna} recently proved that for every $n$-vertex graph $G$
	with maximum degree $\Delta$, sampling $O(\log n)$ colors for each vertex independently
	from $\Delta+1$
	colors suffices for ensuring the existence of a proper coloring of $G$ from the sampled colors
	with probability tending to $1$ as $n \to \infty$. 
	This palette sparsification result
	has received considerable interest and several generalizations and 
	variations of this result have been obtained.
	We obtain a similar result for linear $r$-uniform hypergraphs (i.e. hypergraphs where any two edges have at
	most one common vertex), thus establishing the first palette sparsification
	result for a large family of $r$-uniform hypergraphs. The proof of this result is based on
	a result due to 
	Drgas-Burchardt \cite{Drgas}, which does not appear to be very well-known.
	Additionally, we note that this result does not extend to general $r$-uniform hypergraphs, which answers
	a question from \cite{Casselgren5}.

\section{Single conflict coloring of hypergraphs}

	Although, to the best of our knowledge,
	single conflict coloring has previously only been studied in the setting of graphs,
	it translates naturally to the setting of hypergraphs. 
	A formal definition is as follows.
	Let $H$ be an $r$-uniform 
	hypergraph and denote by $E(v)$ the set of edges containing $v$. 
	A {\em local $k$-partition} of $H$ is a collection $\{c_v\}_{v \in V(H)}$ of maps 
	$c_v : E(v) \to \{1,\dots, k\}$. We think of $c$ as a conflict map, and for each
	edge $e = \{u_1, \dots, u_r\} \in E(H)$, $c$ yields a conflict-$r$-tuple 
	$(c_{u_1}(e), \dots, c_{u_r}(e))$. Given such a local $k$-partition $\{c_v\}$,
	$H$ is {\em conflict $\{c_v\}$-colorable} if there is a map 
	$\varphi : V(H) \to \{1,\dots, k\}$ such that no edge $e = \{u_1, \dots, u_r\}$ 
	satisfies that $\varphi(u_i) = c_{u_i}(e)$ for every $i=1,\dots, r$.
	The {\em single conflict chromatic number} $\single(H)$ of 
	$H$ is the smallest $k$ such that
	$H$ is conflict $\{c_v\}$-colorable for every local $k$-partition $\{c_v\}$.
 	
	Our main goal in this section is to present analogues for hypergraphs of several results
	previously obtained for graphs. We stress that in most cases the proofs are very similar
	to the ones in the graph setting, so we shall generally omit them.

	Alon \cite{Alon} proved that the list chromatic number of a graph with average
	degree $d$ is of order $\Omega(\log d)$, and a similar lower bound for separation choosability
	follows from results of 
	Esperet et al.~\cite{EsperetKangThomasse} and Kwan et al. \cite{Kwan}.
		This implies that both separation choosability, adaptable choosability and
	the single conflict chromatic number grows with the average degree of a graph.
	Indeed, for the latter parameter, the significantly better lower bound
	$\Omega(\sqrt{d/\log d})$ was established by Dvorak et al.~\cite{DvorakEsperetKangOzeki}.

	For hypergraphs the situation is quite different.	
	Let us first recall that a hypergraph is {\em linear (or simple)} if each pair of vertices
	appears in at most one common edge. It has been observed (see e.g.\ 
	\cite{SaxtonThomason}) that for hypergraphs the
	list-chromatic number does not necessarily grow with the average degree, but for
	linear hypergraphs, the list-chromatic number does in fact grow with
	the average degree \cite{AlonKostochka, SaxtonThomason}.

	Since both adaptable choosability and separation choosabilitiy are
	trivially at most the choosability for hypergraphs, this means that separation choosability and adaptable
	choosability do not grow
	with the average degree. Here we shall verify that the lower bound for single conflict
	coloring established in  \cite{DvorakEsperetKangOzeki} remains valid in the
	setting of hypergraphs. The following proposition implies that for hypergraphs, 
	this parameter shows a
	quite different behavior compared to ordinary choosability, and the aforementioned
	graph coloring invariants. The proof is essentially the
	same as in \cite{DvorakEsperetKangOzeki}, which in turn uses an idea
	of Bernstheyn \cite{Bernstheyn}, so we omit it.

\begin{proposition}
\label{prop:hyplowerbound}
	For any $r$-uniform hypergraph $H$ of average degree $d$,
	$\single(H) \geq \left\lceil \left(\frac{d}{\log d}\right)^{1/r}\right\rceil$.
\end{proposition}

	Next, we note that a Brooks-type bound obtained for graphs
	\cite{DvorakEsperetKangOzeki} extend to the setting of hypergraphs.
	This can be proved by a standard application of the Lovasz local lemma as
	is done for similar graph coloring
	invariants \cite{DvorakEsperetKangOzeki, KostochkaZhu, KratochvilTuzaVoigt}.
	The bound for graphs in \cite{DvorakEsperetKangOzeki} can be
	slightly improved bound using the cut lemma due to
	Bernshteyn \cite{Bernshteyn2}, as remarked  in \cite{DvorakEsperetKangOzeki}.
	This approach also works in the setting of $r$-uniform hypergraphs:
	the full proof of the following proposition appears in \cite{Eriksson} and we omit it here.

\begin{proposition}
\label{prop:maxdegree}
	For any $r$-uniform hypergraph $H$ with maximum degree $\Delta=\Delta(H)$,
    $
    \sing(H)\leq \left \lceil r \Delta^{1/r}\right  \rceil. 
    $
\end{proposition}

\newcommand{\K}{\mathcal{K}}
The chromatic number of a complete $n$-order $r$-uniform hypergraph is known to be $\lceil n/(r-1)\rceil$. 
Moreover, in \cite{FurediKostochkaKumbhat} results on the separation choosability and choosability of 
complete $r$-uniform balanced $r$-partite hypergraphs were obtained, where an $r$-uniform hypergraph
is \emph{$k$-partite} if the vertex set $V(H)$ can be partitioned into $k$ sets $V_1,\dots, V_k$ so that each edge contains at most one vertex from each set. Indeed, for such a complete  $r$-partite $r$-uniform hypergraph $H$ with parts of size $n$, we have that $\ch(H) = (1+o(1))\log_r n$.
These facts along with the two preceding propositions, implies the following
hypergraph version of a proposition observed for ordinary graphs in \cite{CasselgrenEriksson};
for details, see \cite{Eriksson}.

\begin{proposition}
    For every positive integer $k$ and every $r\geq 2$, there are $r$-uniform hypergraphs $G_k$ and $H_k$ such that $\sing(G_k)-\ch(G_k)\geq k$, and $\ch(H_k)-\sing(H_k)\geq k$.
\end{proposition}

We remark that 
it appears to be an open question how large the difference between $\ch_{ad}(H)$ and $\ch_{sep}(H)$ can be
for an $r$-uniform hypergraph $H$, as is also the case for ordinary graphs \cite{CasselgrenEriksson}.

Furthermore, since  $\single(H) \leq \chi_{DP}(H)$ holds for any hypergraph $H$, 
Proposition \ref{prop:hyplowerbound}
yields a lower bound on the correspondence chromatic number of a hypergraph. 
We note that
this bound can be slightly improved using a virtually identical proof.

\begin{proposition}
	For any $r$-uniform hypergraph $H$ of average degree $d$,
	$\chi_{DP}(H) \geq \left\lceil \left(\frac{d}{r \log d}\right)^{1/(r-1)}\right\rceil$.
\end{proposition}

	Next, we compare the single conflict chromatic number to adaptable and separation choosability.
	For the case of graphs, the following inequalities were proved in \cite{DvorakEsperetKangOzeki};
	indeed the latter inequality has been
	observed in several places in the literature 
	\cite{EsperetKangThomasse, CasselgrenGranholmRaspaud, DvorakEsperetKangOzeki}.

\begin{proposition}
\label{prop:hyperinequality}
	For any $r$-uniform hypergraph $H$,
	$\single(H) \geq \ch_{ad}(H) \geq \ch_{sep}(H)$.
\end{proposition}	
\begin{proof}
	Let us begin by proving the first inequality. Assume that $k = \single(H)$ and consider
	a $k$-list assignment $L$ for $H$ together 
	with a coloring $\alpha$ of the edges of
	$H$. We define a local $k$-partition $\{c_v\}_{v \in V(H)}$ of $H$ by setting $c_v(e) = \alpha(e)$ 
	if $\alpha(e) \in L(v)$, for every
	edge $e$ of $H$. (Otherwise, define $c_v(e)$ to be an arbitrary integer from
	$L(v)$.) Then, the image of every $c_v$ is a set of size at most $k$, and by
	bijectively mapping this set 
	to $\{1,\dots, k\}$, for every $v \in V(H)$, 
	we may assume that
	$\{c_v\}_{v\in V(H)}$ is a local $k$-partition. By the choice of $k$, there is
	a conflict $\{c_v\}$-coloring of $H$. Clearly, this coloring corresponds to an
	$L$-coloring adapted to $\alpha$. Hence, $\single(H) \geq \ch_{ad}(H)$.

	Next let us prove that $\ch_{ad}(H) \geq \ch_{sep}(H)$. Assume that
	$\ch_{ad}(H) =k$, and consider a list assignment $L$ for $H$ so that
	$|L(u) \cap L(v)| \leq 1$ for any two vertices $u,v$ contained in the same edge.
	Consider an edge $e$ where not all vertices contains a common color in their lists. Note that
	such an edge is not monochromatic in any $L$-coloring of $H$. Thus, we remove
	any such edge from $H$ and consider the obtained hypergraph $H'$. 
	Then any proper
	$L$-coloring of $H'$ is a proper $L$-coloring of $H$.

	On the other hand, for any edge $e$ where all vertices contains a common color 
	$a$, we color $e$ by the color $a$. Since there is at most one such color per edge
	in $H'$, this yields a coloring $\alpha$ of the edges of $H'$. Now, by the choice of
	$k$, there is an $L$-coloring $\varphi$ of $H'$ that is adapted to $\alpha$. It
	is straightforward that $\varphi$ is a proper $L$-coloring of $H$.
\end{proof}

Now, by contrast, let us  describe a family of hypergraphs where the chromatic numbers of
all the considered variants of list coloring coincide.
Our proof of the following proposition is based on a construction by 
	Kostochka and Zhu
	\cite{KostochkaZhu}.
	
	\begin{proposition}
	\label{prop:eq}
		For all positive integers $r,d \geq 2$,
		there is an $r$-uniform $r$-partite hypergraph $H_{r,d}$ such that 
		$$\ch_{sep}(H_{r,d}) = \ch_{ad}(H_{r,d}) = 
		\single(H_{r,d})= \ch(H_{r,d}) = \chi_{DP}(H_{r,d})=d.$$
	\end{proposition}
	
	We shall need the following lemma, which e.g follows from a main result of
	\cite{FurediKostochkaKumbhat}.
	\begin{proposition}
	\label{lem:large}
		For all positive integers $k,r \geq 2$, there is 
		an $r$-uniform $r$-partite hypergraph $H_k$
		with $\ch_{sep}(H_k) \geq k$.
	\end{proposition}
	In \cite{DvorakEsperetKangOzeki} it was noted that every $d$-degenerate
	graph $G$ satisfies that $\single(G) \leq d+1$. This clearly holds in the setting of
	hypergraphs as well (for all the chromatic invariants in Proposition \ref{prop:eq}).
	\begin{proposition}
		For every $d$-degenerate hypergraph $H$, $\single(H) \leq d+1$.	
	\end{proposition}
	
	Using these facts we can prove Proposition \ref{prop:eq}.
	\begin{proof}[Proof of Proposition \ref{prop:eq}]
		We prove that for every $d$, there is a $(d-1)$-denegerate
		$r$-uniform hypergraph $H_{r,d}$ with $\ch_{sep}(H_{r,d}) = d$.
		Since such a $(d-1)$-degenerate hypergraph satisfies that
		$\ch_{sep}(H_{r,d}) = \ch_{ad}(H_{r,d}) = 
		\single(H_{r,d})= \ch(H_{r,d}) = \chi_{DP}(H_{r,d})= d$, this suffices for proving 
		the theorem.
		
		We proceed by induction on $d$. The case $d=2$ is trivial.
		Assume that the proposition is true for degeneracy $d-1 \geq 2$, and
		let $H$ be an $r$-uniform $(d-1)$-degenerate $r$-uniform 
		hypergraph with $\ch_{sep}(H)=d$. We set $n = |V(H)|$
	and $r' = n(r-1)$.

	We construct
	an $r$-uniform $d$-degenerate hypergraph $G'$ 
	with $\ch_{sep}(G')=d+1$ as follows. 
	By Proposition \ref{lem:large}, there is an $r'$-uniform $r'$-partite hypergraph $G$
	with $\ch_{sep}(G) \geq d+1$.
	For each edge $f = \{v_1,v_2, \dots, v_{r'}\}$ of $G$, we add
	a disjoint copy of $H$, denoted $H_f$, with vertex set
	$\{u_{1,f}, \dots, u_{n,f}\}$. We partition the vertices in $f$
	into $n$ parts $X_1,\dots, X_n$, each containing $r-1$ vertices,
	by letting $X_1$ consist of the first $r-1$ parts of $G$, $X_2$ consist of the next $r-1$ parts of $G$, etc.
	The graph $G'$ is now formed from the copies of $H$, and the
	vertices of $G$, by including the edges $\{u_{j,f}\} \cup X_j$
	for $j=1,\dots, n$, in $G'$.
	Since $H$ is $(d-1)$-degenerate, $G'$ is $d$-degenerate.

	Now, since $\ch_{sep}(H) = d$, there is a separated 
	$(d-1)$-list assignment $L_H$
	of $H$ for which there is no proper $L_H$-coloring. 
	Similarly, there is a separated
	$d$-list assignment $L_G$ for $G$ with no proper $L_G$-coloring.
	We may assume that for every edge $f$ of $G$, the
	vertices of $f$ contain exactly one common color $a_f$ in their lists
	(since otherwise we may discard the edge $f$ from $G$).
	Moreover, by permuting the colors for the list assignments
	$L_H$ and $L_G$, we may assume that they use disjoint sets of colors.

	We define a separated list assignment $L$ of $G'$ in the following way.
	Consider the set of edges $f$ where all vertices in $f$ contain the common color $a_f$ in their lists.
	Let $b^f_1,\dots, b^f_n$ be a set of colors  used by neither $L_G$ nor $L_H$.
	For $v \in f$ contained in $X_i$, 
	we set $$L(v)=L_G(v) \cup \{b^f_i \} \setminus \{a_f\}.$$
	For a vertex $u_{i,f} \in V(H_f)$ contained in an edge of the form 
	$\{u_{i,f}\} \cup X_i$, we set $L(u_{i,f}) = L_H(v) \cup \{b^f_i\}$.
	By repeating this procedure for every color used by $L_G$, and updating the list assignment $L$ when
	processing new colors used by $L_G$,
	we obtain a separated $d$-list assignment $L$ for $G'$.
	
	Let us prove that $G'$ is not $L$-colorable. Assume to the contrary, that
	there is an $L$-coloring $\varphi'$ of $G'$. 
	Let $\varphi$ be the $L_G$-coloring of $G$ induced by $\varphi'$ (where a vertex of $f$ that is in $X_i$ 
	thus is colored $a_f$ under $\varphi$ instead of $b^f_i$).
	Now, since $G$ is not separation $d$-choosable,
	there is some edge $f =\{v_1,v_2, \dots, v_{r'}\}$ of
	$G$ for which all vertices gets the same color $a_f$ under $\varphi$.
	Under $\varphi'$, this means that all vertices in the part $X_i$ of
	$f$ are colored $b^f_i$, $i=1,\dots, n$.
	Now, if $\varphi'(u_{i,f}) = b^f_i$ for some $i$, then there is a monochromatic
	edge under $\varphi'$, a contradiction. On the other hand, if
	there is no $i$ such that
	$\varphi'(u_{i,f})= b^f_i$, then this implies that there is an $L_H$-coloring
	of $H$, a contradiction in this case as well.
	In conclusion, $G'$ is $d$-degenerate and not separation $d$-choosable,
	which proves the theorem.
	\end{proof}

	We note that the following problem remains open.

\begin{problem}
	For every positive integer $d$, is there a non-$d$-degenerate hypergraph where 
	$\ch_{sep}(H_{k}) = \ch_{ad}(H_{k}) = 
		\single(H_{k})= \ch(H_{k}) = \chi_{DP}(H) =d+1$?
\end{problem}

The following is also open to the best of our knowledge, as mentioned above.

\begin{problem}
	 For every $r, k \geq 1$, is there an $r$-uniform hypergraph $H_{r,k}$ such that
	$\ch_{ad}(H_{r,k}) - \ch_{sep}(H_{r.k}) \geq k$?
\end{problem}

	As noted above, the choice number of an $r$-uniform hypergraph does not grow with the
	average degree, but this is the case for linear $r$-uniform hypergraphs. 
	Our final question
	is whether this also holds for separation choosability and adaptable choosability.

\begin{problem}
	Is it true that that there are increasing functions $f,g$ such that 
	for a linear $r$-uniform hypergraph $H$ of average degree $d$, $\ch_{sep}(H) = \Omega(f(d))$
	and $\ch_{ad}(H) = \Omega(g(d))$?
\end{problem}



\section{A probablistic model of single conflict coloring}
	
	In \cite{BernshteynDominikKaulMudrock}, the authors introduced
	a random model of DP-coloring of graphs. Inspired by this study we introduce a similar
	model for single conflict coloring of $r$-uniform hypergraphs. 
	We do not introduce the model from 
	\cite{BernshteynDominikKaulMudrock} formally here, but refer
	to that paper for details. Moreover, our formal introduction
	of probabilistic single conflict coloring 'translates' naturally to the setting of probabilistic DP-coloring,
	where, informally speaking, every edge instead yields $k$ possible conflict pairs.

First, we formally define what we mean by a random local $k$-partition.

\begin{definition}[Random local $k$-partition]
    Suppose $H$ is an $r$-uniform hypergraph with $n$ vertices and $m$ edges. For each edge 
	$e=\{v_1, \dots, v_r\}\in E(H)$, let the $r$-tuple $c(e)=(c_{v_1}(e), \dots, c_{v_r}(e))\in [k]^r$ 
	be chosen uniformly at random, 
	independently of the choice for the other edges in $H$. The resulting local 
	$k$-partition $\cv$ of $H$ is called  a \textbf{random local $k$-partition.}
\end{definition}

Given a random local $k$-partition $\{c_v\}$ of $H$, we say that $H$ is \emph{single conflict $k$-colorable with probability $p$} if $\P[H \text{ is conflict }\{c_v\}\text{-colorable}]=p$. 
Typically, we shall be interested in conditions on $k$ that imply that an $n$-vertex $r$-uniform hypergraph is single
conflict $k$-colorable with probability tending to $1$ as $n \to \infty$. More generally,
    an event $A_n$ occurs 
	\emph{with high probability} 
    (abbreviated {\em whp})
	if $\lim_{n \to \infty} \mathbb{P}[A_n] = 1$.

In the following, we shall derive several such analogues of results obtained in \cite{BernshteynDominikKaulMudrock}.
We begin by stating the following analogue of Proposition 1.1 in \cite{BernshteynDominikKaulMudrock},
which gives a lower bound on $k$ for a hypergraph to be single conflict $k$-colorable whp.
In the following, $\rho(H)$ is the maximum 
density of $G$, i.e.\ $\rho(H)=\max_{\emptyset \neq H\subseteq G}|E(H)|/|V(H)|$.

\begin{proposition} 
\label{Proposition: Near 0 prob for sing}
    Let $\varepsilon>0$ and $H$ be a nonempty $r$-uniform hypergraph with maximum density $\rho(H)\geq \exp(1/\varepsilon^r)$. If 
    $$
    1\leq k\leq \left(\frac{\rho(H)}{\log \rho(H)}\right)^{1/r},
    $$ 
    then $H$ is single conflict $k$-colorable with probability at most $\varepsilon$.
\end{proposition}

For $r=2$, Proposition \ref{Proposition: Near 0 prob for sing} implies that a graph $G$ is single conflict $k$-colorable with probability at most $\varepsilon$, given that $\rho(G)\geq \exp(1/\varepsilon^2)$, and 
$$
k\leq \sqrt{\frac{\rho(G)}{\log\rho(G)}}.
$$ 
The analogous bound for DP-colorability from a random
cover with probability at most $\varepsilon$ in \cite{BernshteynDominikKaulMudrock}
is $\frac{\rho(G)}{\log \rho(G)}$.

The proof of Proposition  \ref{Proposition: Near 0 prob for sing} is a simple
first moment argument, similar to the proof of Proposition 1.1 in \cite{BernshteynDominikKaulMudrock}.
To give the reader some intuition of our probabilistic single conflict coloring model, we provide a brief sketch.
In the proof, it will be of importance that $\rho(H)$ can be assumed to be the density of $H$, that is, $\rho(H)=|E(H)|/|V(H)|$. This follows from the observation below, the proof of which we omit
(see \cite{BernshteynDominikKaulMudrock} or \cite{Eriksson} for details). 
Let $p(H,k)=\P[H \text{ is single conflict }k\text{-colorable}]$, for an $r$-uniform hypergraph $H$.
\begin{observation} 
\label{Observation: p'>p}
    If $H'$ is an $r$-uniform subhypergraph of $H$, then for all $k\geq 1$, we have
    $$
    p(H,k)\leq p(H',k).
    $$ 
\end{observation}
\begin{proof}[Proof of Proposition \ref{Proposition: Near 0 prob for sing}]
    Let $\varepsilon>0$, fix $k\leq (\rho(H)/\log \rho(H))^{1/r}$ and let $\rho=\rho(H)$ for an $r$-uniform hypergraph $H$ with $n$ vertices and $m$ edges. By Observation \ref{Observation: p'>p}, we may assume that $\rho=m/n$.

    Let $\{c_v\}$ be a random local $k$-partition of $H$, and enumerate all vertex colorings $\{\varphi_i\}_1^{k^n}$ of $H$. Let $E_i$ be the event that $\varphi_i$ is a conflict $\{c_v\}$-coloring and $X_i$ the indicator random variable for $E_i$. Set $X=\sum_{i\in [k^n]}X_i $; then
    $\E[X_i] 
    =(1-1/k^r)^m$,
    so the expected number of conflict $\{c_v\}$-colorings of $H$ is
    \begin{align*}
    \E[X] 
	&=k^n(1-1/k^r)^m\\
	&=\exp(n(\log k+\rho \log(1-1/k^r))) \\
    	&\leq \exp(n(\log k-\rho/k^r)).
    \end{align*}
    
    Since $k\leq (\rho/\log \rho)^{1/r}$, we have 
	$\log k \leq \frac{1}{r}\log \rho-\frac{1}{r}\log \log \rho$, which implies that
    \begin{align*}
    \log k - \rho/k^r 
    &\leq -\frac{1}{r}\log \log \rho,
    \end{align*}
    so
    $\E[X] 
    \leq \exp\left(-\frac{1}{r}\log \log \rho\right)$.
    Invoking that $\rho\geq \exp(1/\varepsilon^r)$, we deduce
	that this quantity is at most $\varepsilon$, so the desired result follows
	from Markov's inequality.
\end{proof}


Next, we shall prove an analogue of a result in \cite{BernshteynDominikKaulMudrock}, 
relating
single conflict colorability of hypergraphs to the degeneracy.
Consider the following procedure, which is a hypergraph version of a simple algorithm described in
\cite{BernshteynDominikKaulMudrock}.

Let $H$ be an $r$-uniform hypergraph on $n$ vertices with an ordering of the vertices $(v_i)_{i\in [n]}$ and let $d_i^-$ be the \emph{back degree of $v_i$}, which is equal to the degree of $v_i$ in the hypergraph induced by vertices whose indices are less than or equal to $i$. If $H$ is $d$-degenerate, there is an ordering such that $d_i^-\leq d$ for all $i\in [n]$.

\begin{procedure} 
\label{Procedure: (Hypergraphs) Greedyprocedure}
Our input is a $d$-degenerate $r$-uniform hypergraph $H$ and a local 
$k$-partition $\{c_v\}$ of $H$. Let $n=|V(H)|$.
\begin{enumerate}
    \item Order the vertices $V(H)$ as $(v_i)_{i\in [n]}$ so that $d_i^-\leq d$ for all $i\in [n]$.
    \item Set $\varphi(v_1)=1$ and repeat as $i$ goes from $2$ to $n$:
    \begin{enumerate}
        \item Let $L(v_i)=[k]\setminus \{c_{v_i}(e)\setdiv \varphi(u)=c_{u}(e) \text{ for all } u\in e\setminus\{v_i\}\}$.
        \item If $L(v_i)\neq \emptyset$, let $\varphi(v_i)$ be the smallest integer in $L(v_i)$. Otherwise, let $\varphi(v_i)=1$.
    \end{enumerate}
    \item Output $\varphi$.
\end{enumerate}
\end{procedure}

\vspace{\baselineskip}

Next, as in \cite{BernshteynDominikKaulMudrock}, we run Procedure \ref{Procedure: (Hypergraphs) Greedyprocedure} on a random local $k$-partition of $H$. 
For each $i\in [n]$ and $j\in [k]$, let $X_{ij}$ be the indicator random variable of the event $j\in L(v_i)$, that is, $j$ is an available color to choose from at step $i$, and let $Y_{ij}=1-X_{ij}$.

A family of $\{0,1\}$-valued random variables $Z_1,\dots, Z_s$ is {\em negatively correlated} if for every subset
$I \subseteq [s]$, we have
$$\P\left[\bigcap_{i\in I} \{Z_i=1\}\right] \leq \prod_{i \in I} \P[Z_i =1].$$
\begin{lemma}
    \label{Lemma: (Hypergraphs) (i) E and (ii) negatively correlation}
    The following holds for the set of variables $X_{ij}$ and $Y_{ij}$ defined above.
    \begin{enumerate}
        \item[(i)] For all $i\in [n]$ and $j\in [k]$, we have 
        $$
        \E[X_{ij}]\geq \left(1-\frac{1}{k^r}\right)^d.$$
        \item[(ii)] For each $i\in [n]$, the variables $(Y_{ij})_{j\in [k]}$ are negatively correlated.
    \end{enumerate}
\end{lemma}
\begin{proof}
    (i): Fix $i\in [n]$ and let $u_1,\dots, u_t$ be the neighbors of $v_i$ preceding it in the ordering $(v_i)_{i\in [n]}$. For an edge $e\subseteq \{v_i,u_1,\dots, u_t\}$ and a color $j\in[k]$, let $A_{ej}$ be the event that 
    $$
    j=c_{v_i}(e) \text{ and } \varphi(u)=c_{u}(e), \text{ }\forall u \in e\setminus\{v_i\}.
    $$ 
    If $A_{ej}$ occurs for some edge $e$, then the color $j$ is not available to choose at step $i$. The number of edges containing $v_i$ and $r-1$ of the vertices $u_1,\dots, u_t$ is $d_i^-\leq d$. We deduce that
   $\E[X_{ij}]  
               \geq \left(1-\frac{1}{k^r}\right)^{d}$,
    since
    $$
    \P[A_{ej}]=\P[j=c_{v_i}(e) \text{ and } \varphi(u)=c_u(e) \text{ }\forall u\in e\setminus\{v_i\}] =\frac{1}{k^r},
    $$
	because the choice of $\varphi$ on $u_1,\dots, u_t$ is probabilistically independent from
	the choice of conflicts $c_{u}$ for the edge $e$.

    (ii): The proof of this fact is slightly technical and completely identical to the corresponding
	result for DP-coloring of graphs in \cite{BernshteynDominikKaulMudrock}, so we omit it.
\end{proof}

\vspace{\baselineskip}

We shall use the following
Chernoff-type bound for random variables satisfying negative correlation 
\cite{Molloy, PancosiSrinivasan}.

\begin{lemma}
\label{Lemma: C-H bound negative correlation}
    Let $(X_{i})_{i\in [k]}$ be binary random variables. Set $Y_{i}=1-X_i$ and $X=\sum_{i\in [k]}X_i$. If $(Y_i)_{i\in [k]}$ are negatively correlated, then
    $$
    \P[X<\E[X]-t]<\exp\left(-\frac{t^2}{2\E(X)}\right) \quad \text{for all } 0<t\leq \E[X].
    $$
\end{lemma}

We can now prove the following analogue of Theorem 1.3 in \cite{BernshteynDominikKaulMudrock},
which gives a sufficient condition for an $r$-uniform hypergraph to be single conflict $k$-colorable whp.
The proof is essentially the same as in \cite{BernshteynDominikKaulMudrock}; for completeness, we provide
a brief sketch.

\begin{theorem}
    \label{Theorem: d-degenerate r-uniform hypergraphs}
    For all $\varepsilon>0$, there is $n_0\in \Z_+$ and $\delta=\delta(\varepsilon)>0$ such that the following holds. Let $H$ be a $d$-degenerate $r$-uniform hypergraph with $n\geq n_0$ vertices such that 
$d\geq \log ^{1.1/\varepsilon}n$. If
    $$
    k\geq (r+\delta)^{1/r}\left(\frac{d}{\log d}\right)^{1/r},
    $$
    then $H$ is single conflict $k$-colorable with probability at least $1-\varepsilon$.
\end{theorem}

\begin{proof}
    Fix $r$, $\varepsilon>0$, let $c=(r+\delta)^{1/r}$, assume $k$ satisfies the lower bound in the theorem, and let
	$X_i = \sum_{j \in [k]} X_{i,j}$. 
    
    By Lemma \ref{Lemma: (Hypergraphs) (i) E and (ii) negatively correlation} (i), we have
    $$
    \E[X_i]\geq k\left(1-\frac{1}{k^r}\right)^{d}\geq k\exp\left(-\frac{d}{k^r}(1+o(1))\right).
    $$
    Using the fact that $k\geq c(d/\log d)^{1/r}$, we get
    $$
    \E[X_i]\geq c\left(\frac{d}{\log d}\right)^{1/r}\exp\left(-\frac{1+o(1)}{c^r}\log d\right) =\frac{c}{\log^{1/r}d}\cdot d^{\frac{1}{r}-\frac{1+o(1)}{c^r}}>d^{99\epsilon/100},
    $$
    because
    $$
    \frac{1}{r}-\frac{1+o(1)}{c^r}= 
    \frac{\delta-o(1)}{c^r}>\varepsilon,
    $$
    for large enough $n$, 
	and
    $$
    \frac{\log^{1/r}d}{c}< d^{\varepsilon/100}, 
    $$
    which holds for large enough $n$. 
	
	Using Lemma \ref{Lemma: C-H bound negative correlation} and part (ii) 
	of Lemma \ref{Lemma: (Hypergraphs) (i) E and (ii) negatively correlation},
	we deduce that 
$$\P[X_i=0] \leq \P[X_i \leq \E[X_i]/2] \leq \exp \left(-\frac{d^{99\epsilon/100}}{8}  \right) < \varepsilon/n,$$
and by a union bound, it thus holds that
$\P[X_i = 0 \text{ for some $i$}]\leq \varepsilon$.
\end{proof}

For $r=2$, we roughly get the bound $k \geq 2 \sqrt{\frac{d}{\log d}}$ 
which ensures single $k$-conflict colorability
for ordinary graphs,
compared to the corresponding bound $k \geq  2 \left(\frac{d}{\log d}\right)$ for random
DP-colorability in \cite{BernshteynDominikKaulMudrock}.

\begin{corollary}
\label{Theorem: d-degenerate simple graphs}
Let $G$ be an $n$-vertex $d$-degenerate graph with $d\geq \log^{1.1/\varepsilon}n$, where $\varepsilon>0$. If
$$
k\geq \sqrt{(2+\varepsilon)}\sqrt{\frac{d}{\log d}},
$$
then whp $G$ is single conflict $k$-colorable. 
\end{corollary}

We remark that the lower bound on $d$ in Theorem \ref{Theorem: d-degenerate r-uniform hypergraphs} and
Corollary \ref{Theorem: d-degenerate simple graphs} is necessary. This can be verified as in
\cite{BernshteynDominikKaulMudrock} (see Proposition 1.5).
Moreover,
similarly to \cite{BernshteynDominikKaulMudrock}, we believe that it is possible to prove
a better lower bound on $k$ for dense graphs. Specifically, we conjecture that
for a sufficiently dense $r$-uniform hypergraph $H$ 
the property of being single conflict colorable has a sharp threshold
at $\left(\frac{\rho(H)}{\log \rho(H)}\right)^{1/r}$.

For an $n$-vertex complete $r$-uniform hypergraph this suggests that
$k=\left(\frac{n^{r-1}}{r! \log n} \right)^{1/r}$ is
the threshold for single conflict colorability. We prove that
this is true for the case of ordinary graphs.


\begin{theorem} 
\label{Theorem: P close to 1 for simple graphs}
    For all $\varepsilon \in (0,1/2)$, there is $n_0\in \Z_+$ and $\delta=\delta(\varepsilon)>0$ such that the following holds. Suppose $G$ is a graph with $n\geq n_0$ vertices, $m\geq (1-o(1/\log ^2n)) n^2/2$ edges, and maximum density $\rho(G)=m/n$. If
    $$k\geq  (1+\delta)\sqrt{\frac{\rho(G)}{\log \rho(G)}},$$
    then $G$ is single conflict $k$-colorable with probability at least $1-\varepsilon$.
\end{theorem}

Again, the proof of this theorem is inspired by a result in \cite{BernshteynDominikKaulMudrock}.
In particular, we stress that the overall approach and many calculations are similar to the proof of 
Theorem 1.2 in that paper. However, the detailed confirmation of
convergence requires different arguments than the ones
in \cite{BernshteynDominikKaulMudrock}; thus, although
somewhat lengthy and heavily inspired by the proof in \cite{BernshteynDominikKaulMudrock},
we choose to give a complete proof of this theorem.

\begin{proof}
    Let $\varepsilon>0$ and consider a graph $G$ with $n$ vertices, $m= (1-o(1/\log^2n))n^2/2$ edges, and maximum density $\rho =\rho(G)$. Let $k=(1+\delta)\sqrt{\rho/\log \rho}$.

	Let $\{c_v\}$ be a random local $k$-partition and enumerate all vertex $k$-colorings 
	$\{\varphi_i\}_{i=1}^{k^n}$ of $G$. Let $E_i$ be the event that $\varphi_i$ is a conflict $\{c_v\}$-coloring and $X_i$ its indicator random variable, for $i\in [k^n]$, and set $X=\sum_{i=1}^{k^n}X_i$. 
We shall use Chebyshev's inequality in the following form to bound $\P[X>0]$ from below:
    $$
    \P[X>0]=1-\P[X=0]\geq 1-\frac{\Var[X]}{\E[X]^2}.
    $$
    Our task is thus to show that $\Var[X]/\E[X]^2 =o(1)$.    
	Notice that
    \begin{align*}
        \Var[X]=\sum_{(i,j)}\text{Cov}[X_i, X_j] 
        = \sum_{(i,j)} \left( \E[X_iX_j]-\E[X_i]\E[X_j]\right)
        =\sum_{(i,j)} \E[X_iX_j]-\E[X]^2,
    \end{align*}
    where the sums are taken over $(i,j)\in [k^n]^2$.
    The expectation of $X$ is
    $$
    \E[X]
    =k^n(1-1/k^2)^m=k^np^m, \quad \text{with }p:=1-1/k^2.
    $$
    Next, we compute $\E[X_iX_j]$. For an edge $uv\in E(G)$, 
    define, for $\ell \in \{i,j\}$, $A^\ell_{uv}$ to be the event that $\varphi_\ell(u)=c_{u}(uv)$ and $\varphi_\ell(v)=c_{v}(uv)$ does not hold, and let $E_{uv}=A^i_{uv}\cap A^j_{uv}$. We consider two cases.

    \vspace{\baselineskip}

    \textbf{Case A.} $\varphi_i(u)=\varphi_j(u)$ and $\varphi_i(v)=\varphi_j(v)$: \hfill \\ 
   In this case we have that
 $$
    \P[E_{uv}]=\P[A^i_{uv}]=1-1/k^2=p.
    $$
    
    \vspace{\baselineskip}

    \textbf{Case B.} 
	$\varphi_i(u)\neq \varphi_j(u)$: \hfill\\
In this case we have that 
    $$
    \P[E_{uv}]=\P[(c_u(uv),c_v(uv))\notin \{(\varphi_i(u),\varphi_i(v)), (\varphi_j(u), \varphi_j(v))\}] =\frac{k^2-2}{k^2}=:q.
    $$

    \vspace{\baselineskip}

    \newcommand{\al}{\alpha_{i,j}}
    \newcommand{\be}{\beta_{i,j}}
    \newcommand{\ga}{\gamma_{i,j}}
    Let the number of edges in $G$ satisfying the conditions from Cases A and B be $\al$ and $\be$, respectively. Then, since $\al+\be=m$, we have
    \begin{align*}
        \E[X_iX_j] 
        =p^{\al}q^{\be}
        =p^{2m}\left(1+\frac{1}{k^2-2}\right)^{\al}\left(1-\frac{1}{(k^2-1)^2}\right)^{m}.
    \end{align*}
    Thus, 
    \begin{align*}
        \frac{\Var[X]}{\E[X]^2}&=k^{-2n}p^{-2m} \sum_{(i,j)}\E[X_iX_j]-1\\
        &= k^{-2n}\left(1-\frac{1}{(k^2-1)^2}\right)^m\sum_{(i,j)}\left(1+\frac{1}{k^2-2}\right)^{\al}-1.
    \end{align*}

    \vspace{\baselineskip}

    Given $i$ and $j$, $\al$ is the number of edges $uv$ such that $\varphi_i(u)=\varphi_j(u)$ and 
	$\varphi_i(v)=\varphi_j(v)$. 
    We let $I_i=\{(v,\varphi_i(v)) \setdiv v\in V(G) \}$ and $I_j=\{(v,\varphi_j)\setdiv v\in V(G)\}$. 
Since the density of every subgraph $G$ is at most $\rho$, we conclude that if $|I_i\cap I_j|=\nu$, then
    $$\al \le \mu(\nu):=\min \left\{\binom{\nu}{2}, \rho\nu \right\}\leq  \min\left\{\frac{\nu^2}{2},\rho\nu\right\}.$$
There are $\binom{n}{\nu}k^n(k-1)^{n-\nu}$ pairs of colorings $\varphi_i$ and $\varphi_j$ such that $|I_i\cap I_j|=\nu$; thus
    \begin{align*}
        \frac{\Var[X]}{\E[X]^2} &\leq 
\left(1-\frac{1}{(k^2-1)^2}\right)^m\sum_{\nu=0}^n \left( \binom{n}{\nu}k^{-\nu}\left(1-\frac{1}{k}\right)^{n-\nu}\left(1+\frac{1}{k^2-2}\right)^{\mu(\nu)}\right)-1,
    \end{align*}
    and by defining 
    $$
    g(\nu)=\binom{n}{\nu}k^{-\nu}\left(1-\frac{1}{k}\right)^{n-\nu}\left(1+\frac{1}{k^2-2}\right)^{\mu(\nu)},
    $$ 
    the bound reads
    $$ 
    \frac{\Var[X]}{\E[X]^2}\leq \left(1-\frac{1}{(k^2-1)^2}\right)^{m}\sum_{\nu=0}^ng(\nu)-1.
    $$
We shall prove that $(1-1/(k^2-1)^2)^{m}\sum_{\nu=0}^ng(\nu)\leq 1+o(1)$, by considering several different
cases, depending on the magnitude of $\nu$. In many of the cases, we will use the following two estimates:
    \begin{align}
    \label{Equation: Bound on g(nu), simple graphs}
        g(\nu)\leq \left(\frac{ne^{1+1/k}e^{\mu(\nu)/(k^2-2)\nu}}{k\nu e^{n/k\nu}}\right)^\nu\leq \left((1+o(1))e^{\mu(\nu)/k^2\nu}\right)^\nu,
    \end{align}
and    
\begin{align}
    \label{Equation: Bound on (...)^m, simple graphs}
        \left(1+\frac{1}{(k^2-1)^2}\right)^m\leq 
	(1+o(1))e^{-m/k^4}.
    \end{align}

    \newcommand{\factorinfront}{\left(1+\frac{1}{(k^2-1)^2}\right)^m}

    \vspace{\baselineskip}

    \textbf{Case 1.} $0 \le \nu \le (1-\alpha)n/k$ for some $\alpha\in (0,1)$: \\
    With $\mu(\nu)\leq \frac{\nu(1-\alpha)n}{2k}$ in \eqref{Equation: Bound on g(nu), simple graphs}, we have
    $$
    g(\nu)\leq \exp\left(\nu(1+o(1))\frac{(1-\alpha)n}{2k^3}\right)=:r^\nu,
    $$
and so,
    \begin{align*}
        \sum_{0 \leq \nu \leq (1-\alpha)\frac{n}{k}}g(\nu)
	\leq \sum_{0 \leq \nu \leq (1-\alpha)\frac{n}{k}}r^\nu\leq r^{(1-\alpha)n/k}(1+o(1)).
    \end{align*}
    This, together with \eqref{Equation: Bound on (...)^m, simple graphs}, yield
    \begin{align*}
        \factorinfront \sum_{0 \leq \nu \leq (1-\alpha)\frac{n}{k}}g(\nu) &\le (1+o(1))\exp\left(-\frac{m}{k^4}+\frac{(1-\alpha)^2n^2}{2k^4}(1+o(1))\right) 
        =o(1),
    \end{align*}
    since $\alpha>0$ and $m=(1-o(1))n^2/2$.

    \vspace{\baselineskip}

    \textbf{Case 2.} $(1-\alpha)n/k \le \nu \le n/k -\log n$: \hfill \\
    Using Stirling approximation of $\nu!$, we get
    \begin{align*}
        g(\nu)\le \frac{1}{c\sqrt{\nu}}\left((1+o(1))\exp\left(\frac{1}{2k^2}\left(\frac{n}{k}-\log n\right)\right)\right)^{\nu},
    \end{align*}
    where $c>0$ is a constant. Next, since $\nu\geq (1-\alpha)n/k$ we have
    \begin{align*}
        \frac{1}{c\sqrt{\nu}} 
        \le D\exp\left(-\frac{n}{2k}\right),
    \end{align*}
    where $D=\frac{1}{c\sqrt{1-\alpha}}$. We bound $\sum g(\nu)$ as in Case 1 to get
    \begin{align*}
        \factorinfront \sum_{\frac{n}{k}(1-\alpha)\le \nu \le \frac{n}{k}-\log n}g(\nu) 
	& \le (1+o(1)) D e^{-m/k^4} e^{-n/2k} e^{(n/k-\log n)^2/2k^2}=o(1),
    \end{align*}
    because $k=\Omega(\sqrt{n/\log n})$ and $m=(1-o(1/\log ^2n))n^2/2$.

    \vspace{\baselineskip}

    \textbf{Case 3.} $n/k-\log n\le \nu \le n/k+\log ^{5/2}n$: \hfill \\
    As in Case 1, but with $\nu \le n/k + \log ^{5/2}n$, we get
    \begin{align*}
        \factorinfront\sum_{\frac{n}{k}-\log n\le \nu \le \frac{n}{k}+\log ^{5/2}n} g(\nu) &\le (1+o(1))\exp\left(\frac{m}{2k^4}+\frac{1}{2k^2}\left(\frac{n}{k}+\log^{5/2}n\right)^2\right) \\
        &\le (1+o(1))\exp\left(-\frac{n^2}{2k^4}o(1/\log ^2n)+\frac{n\log^{5/2}n}{k^3}+\frac{\log ^5n}{2k^2}\right) \\
        &=1+o(1).
    \end{align*}

    \vspace{\baselineskip}

    \textbf{Case 4.} $n/k+\log^{5/2}n\le \nu \le (1+\alpha)n/k$ for some $\alpha\in (0,1)$: \hfill \\
    Since $k=\Omega(\sqrt{n/\log n})$ we have for a constant $C>0$
    \begin{align*}
        \nu\geq \frac{n}{k}+\log^{5/2}n  
	=\sqrt{\log n}\left(\frac{\sqrt{n}}{C}+\log^2n\right),
    \end{align*}
    thus, with $\nu\leq (1+\alpha)n/k$ and $D=\frac{C}{1+o(1)}$
    \begin{align*}
        \frac{1}{1+\alpha}\le \frac{n}{k\nu} 
        \leq 
        1-\frac{C\log^2n}{\sqrt{n}+C\log^2n}=1-\frac{D\log^2n}{\sqrt{n}},
    \end{align*}
    where the first inequality implies that $e^{-n/k\nu}\leq 1$. We bound $g(\nu)$ as
    \begin{align*}
        g(\nu)&\leq 
        \left( (1+o(1))\left(1-\frac{D\log^2n}{\sqrt{n}}\right)\exp\left(\frac{n}{2k^3}(1+\alpha)(1+o(1))\right)\right)^\nu,
    \end{align*}
    which we use together with \eqref{Equation: Bound on (...)^m, simple graphs} and $1+x\leq e^x$ to get
    \begin{align*}
        &\factorinfront \sum_{\frac{n}{k}+\log^{5/2}n\le \nu \le \frac{n}{k}(1+\alpha)}g(\nu) \\
        &\leq(1+o(1))\exp\left(-\frac{m}{k^4}-\frac{D n \log^2n}{k \sqrt{n}} (1+\alpha)+\frac{n^2}{2k^4}(1+\alpha)^2(1+o(1))\right) \\
        &=(1+o(1))\exp\left( -\frac{n^2}{2k^4}o(1/\log^2n) -\frac{D(1+\alpha)\sqrt{n}\log^2n}{k} +(2\alpha+\alpha^2)\frac{n^2}{k^4}(1+o(1))\right) \\
        &=(1+o(1))\exp\left(o(1)-D'\log^{5/2}n+E\log^2n\right)
        =o(1),
    \end{align*}
    where $D',E>0$.

    \vspace{\baselineskip}

    \textbf{Case 5.} $(1+\alpha)n/k \le \nu \le n^{1-\varepsilon}$: \hfill \\
    Since $n/k\nu \leq 1/(1+\alpha)<1$, we get that
    $\frac{n}{k\nu}e^{-n/k\nu}\leq (1-\xi)/e$ for some $\xi\in(0,1)$, since $xe^{-x}$ has maximum $1/e$ at $x=1$. With $\mu(\nu)\leq \nu^2$, we write
    $$
    \exp\left(\frac{\mu(\nu)}{k^2-2}\right)\leq (1+o(1))\exp\left(\frac{\nu n^{1-\varepsilon}}{n/\log n}\right)
=(1+o(1))\exp\left(\nu (1+o(1))\right),
    $$
    so $g(\nu)\le \left((1+o(1))(1-\xi)\right)^\nu$ and thus, 
    $$
    \sum_{\frac{n}{k}(1+\alpha)\le \nu \le n^{1-\varepsilon}} g(\nu) \le (1+o(1))\frac{(1-\xi)^{(1+\alpha)\frac{n}{k}}}{\xi}=o(1).
    $$

    \vspace{\baselineskip}

    \textbf{Case 6.} $n^{1-\varepsilon}\le \nu \le n$: \hfill \\
 In this case we have that $\frac{n}{k\nu}\leq \frac{n^{\varepsilon}}{k}$ and 
$e^{-n/k\nu} 
=1+o(1)$. Furthermore, with $k\geq C\sqrt{n/\log n}$ for a constant $C>0$, we have
    \begin{align*}
        \exp\left(\frac{\mu(\nu)}{k^2-2}\right)
        \le(1+o(1))\exp\left(\frac{\nu}{2C^2}\log n\right).
    \end{align*}
    Hence, we get that
    \begin{align*}
        g(\nu)\le \left(e(1+o(1))n^{\varepsilon+1/2C^2-1/2}\sqrt{\log n}) \right)^{\nu} =o(1),
    \end{align*}
    as long as 
    $
    C>\sqrt{\frac{1}{1-2\varepsilon}}
    $. 

\bigskip

Summing up, we conlude from Cases 1-6 that $(1-1/(k^2-1)^2)^{m}\sum_{\nu=0}^ng(\nu)\leq 1+o(1)$,
and thus the theorem follows.
\end{proof}

The following follows from the preceding theorem and Proposition \ref{Proposition: Near 0 prob for sing}.

\begin{corollary}
	\label{cor:complete}
	For the complete graph $K_n$ the property of being single conflict colorable from a 
	random local $k$-partition has a sharp threshold at $k =\sqrt{\frac {n}{2 \log n}}$.
\end{corollary}


\section{Palette sparsification of list coloring of hypergraphs}

	In this section we consider the related problem of 
	palette sparsification for list coloring of hypergraphs.
	The problem of coloring hypergraphs from randomly chosen 
	lists was recently initiated in \cite{Casselgren5},
	but was first studied for the case of ordinary graphs
	in the following setting.
 Assign lists of 
    colors to the vertices of a hypergraph $H = H(n)$ with 
    $n$ vertices by choosing for
    each vertex $v$ its list $L(v)$ independently and 
    uniformly at random
	   from all $k$-subsets of a
	   color set $\mathcal{C} = \{1,2,\dots,\sigma\}$.
	   Such a list assignment is called a 
	   \emph{random $(k,\sigma)$-list assignment} for $H$.
	Given that $k$ is a fixed integer, the question studied in \cite{Casselgren, Casselgren2, Casselgren3, Casselgren4,
	HefetzKrivelevich, Krivelevich1, Krivelevich2} is
	how large $\sigma = \sigma(n)$
	should be
	in order to guarantee that with 
    probability tending to $1$ as $n \to \infty$ 
    there is a proper coloring
    of the vertices of $H$ with colors 
    from the random list assignment. Analogous questions for hypergraphs are studied in \cite{Casselgren5}.

	Alternatively, we can fix $\sigma$ 
    (as a function of $n$), and ask how large $k=k(n)$ should be 
    to ensure that
    the random lists support a proper coloring with probability 
    tending to $1$.
	This question was first studied in
	\cite{AndrenCasselgrenOhman, CasselgrenHaggkvist}, and
	following the breakthrough result by Assadi et al.\ \cite{AssadiChenKhanna}
	this type of ``palette sparsification'' questions have recently
	received considerable attention in the setting of graphs
	\cite{AlonAssadi, Dhawan, JainPham, KahnKenney, KahnKenney2, KangKellyKuhnMethukuOsthus, 
	Keevash,  LuriaSimkin}.
	Analogous questions for hypergraphs are, however, almost completely unexplored so far.
	
	In \cite{Casselgren5}, the first author
	proved a palette sparsification result for complete $r$-uniform hypergraphs.
	Here we obtain a palette sparsification-type result for linear hypergraphs, that is,
	hypergraphs where any two vertices are contained in at most one edge.

\begin{theorem}
\label{th:main}
	Let $H =H(n)$ be a linear $r$-uniform hypergraph on $n$ vertices with maximum degree 
	$\Delta$.
	For every constant $C >  (2^r er)^{1/(r-1)}$, there is a constant $A >0$, such that if 
	$k \geq A (\log n)^{1/r}$, $\sigma \geq C \Delta^{1/(r-1)}$ and
	$L$ is random 
	$(k, \sigma)$-list assignment for $H$, 
	then $H$ is whp $L$-colorable.
\end{theorem}
	
	Before proving this theorem, we briefly discuss this result.
	The lower bound $\Delta^{1/(r-1)}$ on $\sigma$ 
	matches the best known upper bound on the choice number
	of linear $r$-uniform hypergraphs, although this upper bound has been improved for the 
	ordinary chromatic
	number \cite{FriezeMubayi}. 
	It would interesting to investigate whether this lower bound on $\sigma$ for linear 
	$r$-uniform hypergraphs could be improved.

	We note that this theorem does not extend to general hypergraphs, 
	which the following example shows.
	Let  $A,C > 0$ be constants, $\sigma = C \Delta^{1/(r-1)}$, 
	$k = A (\log n)^{1/r}$, and consider a 
	random $(k,\sigma)$-list assignment for an $r$-uniform hypergraph $H=H(n)$ of order $n$ that 
	is a disjoint union of $\lfloor n/(k(r-1)+1)) \rfloor$ complete hypergraphs of order $k(r-1)+1$
	(and possibly some isolated vertices). 
	Note that the maximum degree $\Delta$ in $H$ is 
	$$\Delta= \binom{k(r-1)+1}{r-1} \sim \frac{(k(r-1)+1)^{r-1}}{(r-1)!},$$
	so $\sigma \sim \frac{C k(r-1)}{((r-1)!)^{1/(r-1)}}$.

Now, since $H$ is a disjoint union of $(k(r-1)+1)$-cliques, it is $L$-colorable if and only if
there is no clique such that all vertices get the same color list. 
Now, the probability that
all vertices in such a clique get the same list is 
$$\binom{\sigma}{k}^{-k(r-1)} \geq 2^{-Ck^2(r-1)}.$$
Thus the probability that there is no such clique in $H$ is at most 
$$
\left(1-2^{-Ck^2(r-1)}\right)^{\frac{n}{k(r-1)+1}}
\leq
\exp\left(-2^{-Ck^2(r-1)}\frac{n}{k(r-1)+1} \right),
$$
which tends to zero as $n \to \infty$, for any choice of the constants $A$ and $C$. 
Thus, whp $H$ is not $L$-colorable. In conclusion, Theorem \ref{th:main} does not extend to general
$r$-uniform hypergraphs, which answers a question from \cite{Casselgren5}.

	Finally, let us remark that for general $r$-uniform hypergraphs, it was proved in \cite{Casselgren5}
	that for every constant $C>0$ there is a (relatively small) 
	constant $A$ such that if $$\sigma \geq C \Delta^{1/(r-1)} \log n, \,
	k \geq A \log n,$$ and $L$ is a random $(k,\sigma)$-list assignment for an $n$-vertex $r$-uniform 
	hypergraph $H$ with maximum degree $\Delta$, then $H$ is whp $L$-colorable. 
	It would be interesting to investigate if the factor $\log n$
	could be removed from the bound on $\sigma$, which would be best possible (up to multiplicative constants) 
	by the example of complete $r$-uniform
	hypergraphs \cite{Casselgren5}.

\bigskip

We now turn to the proof of Theorem \ref{th:main}.
We shall use a Reed-type result on hypergraph list coloring involving color degrees
due to Drgas-Burchardt \cite{Drgas}, which does not appear to be very well-known. (Indeed, while working
on the proof of Theorem \ref{th:main}, we first proved a version of Theorem \ref{th:Drgas} below, but later discovered
the result of Drgas-Burchardt while preparing this manuscript.)

First we introduce (and simplify) some terminology from that paper.
For a hypergraph $H$ with a list assignment $L$, we denote by $H_c(L)$ the subgraph of $H$ induced by all edges where every vertex contains the color $c$ in its list. We denote by $d_{c}(v,H,L)$ 
the {\em color degree} of $(c,v)$ in $H$, that is, the number of edges containing
$v$ such that every vertex in these edges has the color $c$ in its list. Similarly, the {\em edge-color degree}
$d_{c}(e, H, L)$ of a pair $(e,c)$, where $e$ is an edge, is the sum
$$d_{c}(e,H,L) = \sum_{v \in e} d_{c}(v,H,L).$$
The {\em color-degree} $d_{col}(e)$ of an edge $e$ is defined as
$$d_{col}(e,L) = \sum_{c \in L(v) : v \in e} d_{c}(e,H,L).$$

We denote by $\Delta_{col}(H,L)$, the maximum color-degree of an edge in $H$. 
Using the Lovasz local lemma, Drgas-Burchardt proved the following theorem \cite{Drgas}.

\begin{theorem}
\label{th:Drgas}
	Let $H$ be an $r$-uniform hypergraph and $L$ a $k$-list assignment for $H$.
	If $k \geq \left(e (\Delta_{col}(H,L) +1) \right)^{1/r}$, then $H$ is $L$-colorable.
\end{theorem}

We shall also use a standard version of Chernoff's inequality for binomially distributed random variables,
see e.g.~\cite{Janson}. 

\begin{theorem}
Let $X= X_1 + \dots + X_n$ be a sum of $n$ identically Bernoulli distributed random variables $X_1,\dots, X_n$.
For any $\varepsilon >0$,
$$\mathbb{P}[X-\mathbb{E}[X] \geq \varepsilon\mathbb{E}[X]] \leq
\exp\left(-\frac{\varepsilon^2}{3}\mathbb{E}[X]\right).$$
\end{theorem}

Using these auxiliary results, we shall prove Theorem \ref{th:main}. Our proof is a modification of a result
in \cite{AlonAssadi} to the setting of hypergraphs.

\begin{proof}[Proof of Theorem \ref{th:main}]
	Let $C > (2^r er)^{1/(r-1)}$, $A$ a ``large'' constant, and consider a  random 
	$(k, \sigma)$-list assignment $L$ for $H$.
	Let $d_c(v,H,L)$ be a random variable counting the color degree of $(c,v)$ in $H$. Then

$$
\mathbb{E}[d_c(v,H,L)] \leq \Delta \frac{\binom{\sigma-1}{k-1}^{r-1}}{\binom{\sigma}{k}^{r-1}}
\leq \Delta \left(\frac{k}{\sigma}\right)^{r-1}.
$$
Since we choose the colors for each vertex independently and $H$ is linear, we can apply Chernoff's inequality to obtain,
$$
\mathbb{P}\left[d_c(v,H,L) > (1 + \varepsilon)\mathbb{E}[d_c(v,H,L)]\right] < 
\exp \left( \frac{-\varepsilon^2 \Delta}{3}\left(\frac{k}{\sigma}\right)^{r-1}\right),
$$
where $\varepsilon>0$ is an arbitrarily small constant. 

We say that a color $c \in L(v)$ is {\em bad} for $v$ if
it satisfies that $$d_c(v,H,L) > (1 + \varepsilon)\mathbb{E}[d_c(v,H,L)].$$
The probability that a vertex $v$ has at least $\alpha k$ bad colors,
where $\alpha>0$ is a small constant, is at most
$$
\mathbb{P}[v \textrm{ has at least $\alpha k$ bad colors}] \leq 
\binom{k}{\alpha k} \exp \left( \frac{-\varepsilon^2 \Delta}{3}\left(\frac{k}{\sigma}\right)^{r-1}\right)^{\alpha k}
\leq 2^k e^{-\frac{\varepsilon^2\alpha k^r}{3C^{r-1}}}.
$$
Now, since $k \geq A (\log n)^{1/r}$, and $A$ is large enough, taking a union bound over all vertices yields that whp every vertex
has less than $\alpha k$ bad colors.

We define a new random list assignment $L'$ by
for every vertex $v$ removing all bad colors from $L(v)$.
Then $|L'(v)| \geq (1-\alpha) k$ for every vertex $v$, and moreover, since every edge contains $r$ vertices, we get that
$$d_{col}(e,L') \leq rk (1+\varepsilon) \Delta \left(\frac{k}{\sigma}\right)^{r-1}
\leq
rk \frac{ (1+\varepsilon) k^{r-1}}{C^{r-1}} = \frac{ r(1+\varepsilon) k^r}{C^{r-1}}.$$
Thus,
$$\left(e (\Delta_{col}(H,L) +1) \right)^{1/r} \leq 
\left(\frac{e  (r(1+\varepsilon) k^{r})}{C^{r-1}}+e\right)^{1/r}.$$
Now,
since $C > (2^r er)^{1/(r-1)}$ and $\varepsilon$ is arbitrarily small, we have that
$$|L'(v)| \geq \left(e (\Delta_{col}(H,L) +1) \right)^{1/r},$$
for every vertex $v$,
so Theorem \ref{th:Drgas} yields that whp $H$ is $L$-colorable.
\end{proof}


{}

\end{document}